\documentclass{amsart}
\usepackage{amsfonts}
\usepackage{amsfonts,amssymb,amsmath,amsthm}
\usepackage{url}
\usepackage{enumerate}
\usepackage{bbm}
\usepackage[all]{xy}

\newtheorem{thrm}{Theorem}[section]
\newtheorem{lem}[thrm]{Lemma}

\newtheorem{cor}[thrm]{Corollary}
\theoremstyle{definition}

\newtheorem{remark}[thrm]{Remark}
\numberwithin{equation}{section}

\author{Zhiqiang Li}
\address{Research Center for Operator Algebras, East China Normal University, Shanghai, China \ 200062}
\address{Chongqing University, Chongqing, China \ 401331}

\email{zqli@cqu.edu.cn}
\author{Wenda zhang}
\address{College of Mathematics and Statistics, Chongqing Jiaotong University, Chongqing, China \ 401224}
\email{wendazhang951@aliyun.com}


\begin{document}

\newcommand{\li}{\lim\limits_{\longrightarrow}}
\newcommand{\Ca}{$\textrm{C}^*$-algebra}
\newcommand{\il}{inductive limit}
\newcommand{\C}{$\textrm{C}^*$}
\newcommand{\K}{$\textrm{K}$}
\newcommand{\KK}{$\textrm{K}\!\textrm{K}$}
\newcommand{\I}{$\textrm{I}$}
\newcommand{\M}{\operatorname{M}}

\title[AF systems of certain cyclic groups]{AF systems of certain cyclic groups}

\begin{abstract}
In this paper, we give a \K-theoretic classification for the \C-dynamical systems
 $\lim\limits_{\longrightarrow}(A_n,\alpha_n,G)$, where each $A_n$ is a finite
 dimensional C*-algebra, and $G$ is a cyclic group of prime order. Consequently, all inductive limit actions of such cyclic groups on AF algebras are classified. Such actions contain natural examples of finite group actions on UHF algebras which do not have the tracial Rokhlin property (see \cite{Phi}).
\end{abstract}
\maketitle

\section{Introduction} \label{sect1}
A number of results concerning the classification of \C-algebras
have been obtained under the Elliott programme. However,
classification of group actions on \C-algebras is still a far less
developed subject, partially because of \K-theoretic difficulties.
Among all the group actions on C*-algebras, there are typical ones ,
i.e., inductive limit group actions, in other words, group actions
are compatible with the inductive limit structuresa of the
C*-algebras. In such a setting, it is possible to classify them
using the equivariant version of Elliott's intertwining argument.

Given a compact group $G$, let $A=\lim\limits_{\longrightarrow}A_n$
be the inductive limit of a sequence of finite dimensional
\C-algebras, let $\alpha=\lim\limits_{\longrightarrow}\alpha_n$ be
an inductive limit action of $G$ on $A$. Then one can form the
C$^*$-algebra cross product $A\rtimes_{\alpha}G=\li
A_n\rtimes_{\alpha_n}G$. If each $\alpha_n$ is given by an inner
automorphism arisen from a unitary representation of the group $G$,
then it was shown in \cite{HR} that the natural \K-theory data of
$A\rtimes_{\alpha}G$ is a complete invariant for the \C-dynamical
system $(A,\alpha,G)$. Such actions are referred to as locally
representable actions. In the case that $A$ is unital, the \K-theory
data in \cite{HR} consists of the \K-group
$\K_0(A\rtimes_{\alpha}G)$ together with (i) the natural order
structure, (ii) the special element coming from the projection given
by averaging the canonical unitaries of the cross product, (iii) the
natural module structure over the representation ring $\K_0(G)$. In
\cite{K}, Kishimoto considered locally representable actions of
finite groups on inductive limit algebras with more complicated
building blocks (circles), and in \cite{BEEK}, this study was
extended to still more complicated inductive limit systems and to
general compact groups, but still requiring local representability.

So it is interesting to consider the case beyond locally
representable actions. Along this line, in \cite{ES}, G. A. Elliott
and H. Su removed this local representability hypothesis in the case
where the group is $Z/2Z$ and the building blocks are finite
dimensional. In \cite{Su}, this local representability condition was
also removed, where the group is still $Z/2Z$, but the inductive
limits are certain real rank zero systems built on some
subhomogeneous graph \Ca s. Then, it is a natural question to ask to
which extent one can obtain a K-theoretic classification for more
general group actions on C*-algebras. Conceivably, actions of finite
abelian groups will be a quite large class. To classify these
actions, from the viewpoint of group structure theory, the $p$
(prime) groups (groups with order being some power of $p$) case will
be fundamental, and among them, the cyclic groups with prime orders
should be the first test case. In the present paper, a \K-theoretic
classification for inductive limit actions of such cyclic groups
on AF (approximately finite dimensional) algebras is obtained.

On the other hand, there is another class of group actions on
C*-algebras which draw many people's attention, namely, the group
actions with the (tracial) Rokhlin property. For the discrete groups
$Z$ and $Z^d$, I. Hirshberg, W. Winter, and J. Zacharias (in
\cite{HWZ}) and N. C. Phillips (in \cite{Phi1}) showed that the
actions with the tracial Rokhlin property are generic for nice
C*-algebras (for example, tracially AF algebras). For finite group
action case, M. Izumi showed that there are serious obstructions for
C*-algebras admitting finite group actions with the Rokhlin
property, (see \cite{izumi1}, \cite{izumi2}). He showed that for a
simple unital C*-algebra $A$, if either K$_0(A)$ or K$_1(A)$ is
isomorphic to $\mathbb{Z}$, then there is no non-trivial finite
group action with the Rokhlin property on $A$ (see Theorem 3.3 and
Theorem 3.6 in \cite{izumi2}). In fact, there are natural examples
of inductive limit actions of cyclic groups on UHF algebras which do
not have the (tracial) Rokhlin property (see \cite{Phi}).  We quote
the example of $Z/2Z$ action here:
$$A=\bigotimes\limits_{n=1}^{\infty}\textrm{M}_{2^n}, \;\alpha=\bigotimes\limits_{n=1}^{\infty}\alpha_n,
\;\alpha_n=Ad\left(
                                                                    \begin{array}{cc}
                                                                      1_{2^{n}-1} & 0 \\
                                                                      0 & -1 \\
                                                                    \end{array}
                                                                  \right),
$$ the distribution of the eigenvalues of the unitaries indicate that this action does not have the tracial Rokhlin property (see \textbf{Example 2.9} in \cite{Phi} in detail). Since such
examples have inductive limit structure, they sit in our
classifiable classes.

Throughout this paper, let us denote the group $Z/pZ$ by $Z_p$,
where $p$ is a prime. We use both $id$ and $I$ to denote the
identity matrix.

To state the invariant, let $A$ be a unital \Ca, and let $\alpha$ be
a group action of $Z_p$ on $A$. The invariants we need are as
follows:

(1)$\,(\K_0(A),\, \K_0(A)^+,[1_A],\, \alpha_*)$,

(2)$\,(\K_0(A\rtimes_\alpha Z_p),\,\K_0(A\rtimes_\alpha
Z_p)^+,\,\zeta,\, \hat{\alpha}_*)$, where $\zeta$ is the special
element in $\K_0(A\rtimes_\alpha Z_p)$ and $\hat{\alpha}$ is the
dual action of $\hat{Z}_p$ on $A\rtimes_\alpha Z_p$,

(3)\;$\iota_{*}:$ $\K_0(A)\rightarrow \;\K_0(A\rtimes_\alpha Z_p)$,
where $\iota$ is the canonical embedding of $A$ into
$A\rtimes_\alpha Z_p$.

\noindent $(1)$ and $(3)$ are necessary, since the action may not be
inner, the information in $\K_0(A)$ may not be recovered completely
from $\K_0(A\rtimes_\alpha Z_p)$, we must adjoin this, as well as
the actions on the \K-groups, to the invariant. We state the main
theorem here.
\begin{thrm}\label{main} Let $(A,\alpha,Z_p)=\li (A_n,\alpha_n,Z_p)$ and
$(B,\beta,Z_p)=\li (B_n,\beta_n,Z_p)$ be two approximately finite
dimensional \il\;\C-dynamical systems, let $F$ be a scaled order
preserving group isomorphism from $(\K_0(A),\alpha_*)$ to
$(\K_0(B),\beta_*)$, and let $\phi$ be an order preserving group
isomorphism from $(\K_0(A\rtimes_\alpha Z_p),\hat{\alpha}_*)$ to
$(\K_0(B\rtimes_\beta Z_p),\hat{\beta}_*)$ mapping the special
element to the special element. Suppose that the following diagram
commutes:
$$
\xymatrix{ K_0(A) \ar[r] \ar[d]_F& K_0(A\rtimes_\alpha Z_p) \ar[d]^{\phi} \\
K_0(B) \ar[r] & K_0(B\rtimes_\beta Z_p).}
$$
Then there is an isomorphism $\psi$ from $(A,\alpha,Z_p)$ to
$(B,\beta,Z_p)$ such that $\psi_*=F$ and such that the extension of
$\psi$ to $A\rtimes_\alpha Z_p$ induces $\phi$.
\end{thrm}

The paper is organized as follows. In Section 2, some preliminaries
are given about the irreducible actions of $Z_p$ on finite
dimensional C*-algebras. In Section 3, a local existence result is
proved, namely, morphisms between the invariant of the finite
dimensional \C-dynamical systems can be lifted to morphisms between
the finite dimensional \C-dynamical systems. In Section 4, a local
uniqueness result is obtained, namely, for any two morphisms between
the finite dimensional \C-dynamical systems, if their induced maps
agree on the invariant, then they are unitarily equivalent by an
equivariant unitary, i.e., a unitary in the fixed point subalgebra
of the codomain algebra. These two results are the main ingredients
in the equivariant Elliott's intertwining argument. In Section 5,
the main theorem will be proved by the equivariant Elliott's
intertwining argument.

\section{preliminaries} \label{sect2}
Let $A=\bigoplus\limits_{k=1}^{m}\M_{n_k}$ be a finite dimensional
\Ca, and  let $\alpha$ be a group action of $Z_p$ on $A$. Since
$Z_p$ is cyclic, then $\alpha$ is determined by the corresponding
automorphism $\rho$ of the generator of $Z_p$. From basic
representation theory, $\alpha$ can be decomposed into a finite
direct sum of irreducible actions. Each irreducible action has the
form either $(\M_n,\rho)$ or
$(\underbrace{\textrm{M}_n\oplus...\oplus \textrm{M}_n}_{p},\rho)$.
Let us prepare all the \K-theoretic information about the
irreducible actions.

In the case $(\M_n,\rho)$, $\rho$ is given by a unitary $V\in \M_n$:
$$\rho(a)=VaV^*,\;a\in \textrm{M}_n,$$ where $V$ satisfies $V^p=I$
and could be chosen to be diagonal.

\begin{lem} $M_n\rtimes_{\alpha}Z_p$ is isomorphic to
$\underbrace{M_n\oplus...\oplus M_n}_{p}$.
\end{lem}
\begin{proof} The identification map is given as follows:\begin{align*}
a_0+a_1U_{\rho}+&a_2U_{\rho^2}+...+a_{p-1}U_{\rho^{p-1}}\\
&\longrightarrow(a_0+a_1V+a_2V^2...+a_{p-1}V^{p-1},\\
&a_0+e^{i\frac{2\pi}{p}}a_1V+e^{i\frac{4\pi}{p}}a_2V^2+...+e^{i\frac{2(p-1)\pi}{p}}a_{p-1}V^{p-1},\\
&a_0+e^{i\frac{4\pi}{p}}a_1V+e^{i\frac{8\pi}{p}}a_2V^2+...+e^{i\frac{2(p-2)\pi}{p}}a_{p-1}V^{p-1},\\
&,...,\\
&a_0+e^{i\frac{2(p-1)\pi}{p}}a_1V+e^{i\frac{2(p-2)\pi}{p}}a_2V^2+...+e^{i\frac{2\pi}{p}}a_{p-1}V^{p-1}),
\end{align*} where $U_{\rho^k},k=1,...,p-1$ are the canonical unitaries in
the cross product algebra. Then one can verify the lemma by this
formula.
\end{proof}

\begin{remark} This lemma is also true if one replaces $\textrm{M}_n$ by an arbitrary unital C*-algebra $A$, and
replaces $Z_p$ by an arbitrary finite cyclic group $G$, i.e., the
map above is still an isomorphism.
\end{remark}

Then $\K_0(\M_n)=Z$, $\K_0(\M_n\rtimes
Z_p)=\underbrace{Z\oplus...\oplus Z}_{p}$, and the map from
$\K_0(\M_n)$ to $\K_0(\M_n\rtimes Z_p)$ sends $x$ to
$(\underbrace{x,...,x}_{p})$; and $\rho_*$ is trivial. It is well
known that $\hat{Z}_p=Z_p$, and the generator of $\hat{Z}_p$ is
$\hat{\rho}$ which takes the identity element to $1$, and takes
$\rho$ to $e^{i\frac{2\pi}{p}}$. So
$$\hat{\rho}(\sum\limits_{k=0}^{p-1}a_kU_{\rho^k})=a_0+e^{-i\frac{2\pi}{p}}a_1U_{\rho}+e^{-i\frac{4\pi}{p}}a_2U_{\rho^2}+...+e^{-i\frac{2(p-1)\pi}{p}}a_{p-1}U_{\rho^{p-1}}$$by the identification formula above, it is easily to see that $\hat{\rho}$ and $\hat{\rho}_*$ is the permutation given by
$(\xi_1,\xi_2,...,\xi_p)\rightarrow(\xi_p,\xi_1,...,\xi_{p-1})$. The
special element $\zeta=(l_0,...,l_{p-1})$, where $l_k$ is the number
of the eigenvalue $e^{i\frac{2k\pi}{p}}$ of the unitary $V$ which
implements the automorphism $\rho$.

In the case $(\underbrace{\textrm{M}_n\oplus...\oplus
\textrm{M}_n}_{p},\rho)$, up to conjugacy, $\rho$ can be chosen to
have the following form:
$\rho(a_1,a_2,...,a_p)=(a_p,a_1,...,a_{p-1}).$

\begin{lem} $\underbrace{\textrm{M}_n\oplus...\oplus
\textrm{M}_n}_{p}\rtimes_{\alpha}Z_p$ is isomorphic to $\M_{pn}$.
\end{lem}
\begin{proof} The identification map is given as follows:\begin{align*}
(a^0_0,a^0_1,...,a^0_{p-1})+&(a^1_0,a^1_1,...,a^1_{p-1})U_{\rho}+...+(a^{p-1}_0,a^{p-1}_1,...,a^{p-1}_{p-1})U_{\rho^{p-1}}\\
&\longrightarrow\left(
   \begin{array}{cccc}
     a^0_0 & a^1_0 & \ldots & a^{p-1}_0 \\
     a^{p-1}_{p-1} & a^0_{p-1} & \ldots & a^{p-2}_{p-1} \\
     \vdots & \vdots &  & \vdots \\
     a^1_1 & a^2_1 & \ldots & a^{0}_1\\
   \end{array}
 \right).
\end{align*}
\end{proof}

Then $\K_0(\underbrace{\textrm{M}_n\oplus...\oplus
\textrm{M}_n}_{p}\rtimes_{\alpha}Z_p)=Z$, the canonical map from
K$_0(\underbrace{\textrm{M}_n\oplus...\oplus \textrm{M}_n}_{p})$ to
K$_0(\textrm{M}_{pn})$ sends $(x_1,...,x_p)$ to
$\sum\limits_{k=1}^{p}x_k$; and $\rho_{*}$ is the permitation. Let
$$\xi=(a^0_0,a^0_1,...,a^0_{p-1})+(a^1_0,a^1_1,...,a^1_{p-1})U_{\rho}+...+(a^{p-1}_0,a^{p-1}_1,...,a^{p-1}_{p-1})U_{\rho^{p-1}},$$

 so
\begin{equation*}
\begin{split}
\hat{\rho}(\xi)&=(a^0_0,a^0_1,...,a^0_{p-1})+e^{-i\frac{2\pi}{p}}(a^1_0,a^1_1,...,a^1_{p-1})U_{\rho}+...\\
&\quad+e^{-i\frac{2(p-1)\pi}{p}}(a^{p-1}_0,a^{p-1}_1,...,a^{p-1}_{p-1})U_{\rho^{p-1}}.
\end{split}
\end{equation*} Then by the identification in Lemma 2.3, the dual action is as
follows:
 $$\hat{\rho}(C)=\left(
              \begin{array}{cccc}
                1 &  &  &  \\
                 & e^{i\frac{2\pi}{p}} &  &  \\
                 &  & \ddots &  \\
                 &  &  & e^{i\frac{2(p-1)\pi}{p}} \\
              \end{array}
            \right)
C\left(
              \begin{array}{cccc}
                1 &  &  &  \\
                 & e^{-i\frac{2\pi}{p}} &  &  \\
                 &  & \ddots &  \\
                 &  &  & e^{-i\frac{2(p-1)\pi}{p}} \\
              \end{array}
            \right),$$
for all $C\in M_{pn}$. Hence $\hat{\rho}_{*}$ is trivial. The
special element $\zeta$ is $n$.
\begin{remark}Lemma 2.3 also verifies the Takai Duality
for $(\M_n, Z_p)$.

\end{remark}
\section{Local existence}\label{sect3}
In this section, we are going to establish the local existence
theorem, which states that morphisms between the invariant of the
finite dimensional \C-dynamical systems can be lifted to morphisms
between the finite dimensional \C-dynamical systems. This local
existence theorem together with the local uniqueness theorem in the
next section are the two main ingredients in the equivariant
Elliott's intertwining argument.

\begin{thrm} Let $(A_k,\alpha_k,Z_p)$ and $(B_n,\beta_n,Z_p)$ be two
irreducible finite dimensional \C-dynamical systems. Let $F_k$ be an
ordered group morphism from $(K_0(A_k),[$ $1_{A_k}],\alpha_{k*})$ to
$(K_0(B_n),$ $[1_{B_n}],\beta_{n*})$. Let $\phi_k$ be an ordered
group morphism from
$(K_0(A_k\rtimes_{\alpha_k}Z_p),\hat{\alpha}_{k*})$ to
$(K_0(B_n\rtimes_{\beta_n}Z_p),\hat{\beta}_{n*})$, which preserves
the special elements. Then there exists a homomorphism $\psi_k$ from
$(A_k,\alpha_k,Z_p)$ to $(B_n,\beta_n,Z_p)$, such that
$\psi_{k*}=F_k$, and $\widetilde{\psi}_{k*}=\phi_k$, where
$\widetilde{\psi}_k$ is the natural extension of $\psi_k$ to
$A_k\rtimes_{\alpha_k}Z_p$.
\end{thrm}
\begin{proof} We are going to prove the theorem in four different cases.

(1). $A_k=\textrm{M}_k,\;B_n=\textrm{M}_n$.

Suppose $U\in \textrm{M}_k$ and $V\in \textrm{M}_n$ are the two
unitaries which implement the automorphisms on $\textrm{M}_k$ and
$\textrm{M}_n$ respectively. Since $F_k$ preserves the scale, then
$F_k=\frac{n}{k}$. By Lemma 2.1, $\phi_k$ is of the
form:$$\phi_k=\left(
                                                             \begin{array}{cccc}
                                                               l_{11} & l_{12} & \ldots & l_{1p} \\
                                                               l_{21} & l_{22} & \ldots & l_{2p} \\
                                                               \cdots & \cdots &  \cdots& \cdots \\
                                                               l_{p1} & l_{p2} & \ldots & l_{pp} \\
                                                             \end{array}
                                                           \right).
$$ Moreover, $\phi_k$ intertwines $\hat{\alpha}_{k*}$ and
$\hat{\beta}_{n*}$, by calculation, one obtains that $$\phi_k=\left(
                                                             \begin{array}{cccc}
                                                               l_{11} & l_{12} & \ldots & l_{1p} \\
                                                               l_{1p} & l_{11} & \ldots & l_{1p-1} \\
                                                               \cdots & \cdots &  \cdots& \cdots \\
                                                               l_{12} & l_{13} & \ldots & l_{11} \\
                                                             \end{array}
                                                           \right).
$$ By assumption, we have $\phi_{k}\zeta=\zeta'$, where $\zeta$ and $\zeta'$ are the
two special elements in $ \textrm{M}_k\rtimes_{\alpha_k}Z_p$ and
$\textrm{M}_n\rtimes_{\beta_n}Z_p$, then
$(l_{11}+l_{12}+...+l_{1p})k=n$.

Define
\begin{align*}
&e_1=e_2=...=e_{l_{11}}=I_k,\\
&e_{l_{11}+1}=e_{l_{11}+2}=...=e_{l_{11}+l_{12}}=e^{-i\frac{2\pi}{p}}\otimes
id_{\frac{n}{k}},\\
&e_{l_{11}+l_{12}+1}=e_{l_{11}+l_{12}+2}=...=e_{l_{11}+l_{12}+l_{13}}=e^{-i\frac{4\pi}{p}}\otimes
id_{\frac{n}{k}},\\
&,...,\\
\end{align*}
set
$$e=diag(e_1,...\,,e_{l_{11}},e_{l_{11}+1},...\,,e_{l_{11}+l_{12}},...\,,e_{l_{11}+...+l_{1p}}),$$
then $e(U\otimes id_{\frac{n}{k}})=(U\otimes id_{\frac{n}{k}})e$.

Because $\phi_k$ preserves the special elements, then the eigenvalue
list of $(U\otimes id_{\frac{n}{k}})e$ is the same as $V$, then
there exists a unitary $W$, such that $W^*VW=(U\otimes
id_{\frac{n}{k}})e$.

Define a homomorphism $\psi_k:\textrm{M}_k\rightarrow \textrm{M}_n$
by:
$$\psi_k(a)=W(a\otimes id_{\frac{n}{k}})W^*,$$
then $(U^*\otimes id_{\frac{n}{k}})W^*VW(a\otimes
id_{\frac{n}{k}})=(a\otimes id_{\frac{n}{k}})(U^*\otimes
id_{\frac{n}{k}})W^*VW$,

\noindent namely, $\psi_k$ intertwines $\alpha_k$ and $\beta_n$, and
$\psi_{k*}=F_k$. Since the natural extension $\widetilde{\psi}_k$
intertwines $\hat{\alpha}_k$ and $\hat{\beta}_n$, by calculation,
$\widetilde{\psi}_{k*}=\phi_k$.

(2). $A_k=\textrm{M}_k,\;B_n=\underbrace{\textrm{M}_n\oplus...\oplus
\textrm{M}_n}_{p}$.

Obviously, $F_k=\left(
                  \begin{array}{c}
                    \frac{n}{k} \\
                    \vdots \\
                    \frac{n}{k} \\
                  \end{array}
                \right)
$. Since $\phi_k$ intertwines $\hat{\alpha}_{k*}$ and
$\hat{\beta}_{n*}$, by calculation, we have: $\phi_k=\left(
                                                       \begin{array}{cccc}
                                                         l_1 & l_1 & \cdots & l_1 \\
                                                       \end{array}
                                                     \right)
$. Moreover, because $\phi_k$ preserves the special element, then
$l_1=\frac{n}{k}$. Let $V$ be the unitary implementing the
automorphism on $\textrm{M}_k$.

Define a homomorphism
$\psi_k:\textrm{M}_k\rightarrow\underbrace{\textrm{M}_n\oplus...\oplus
\textrm{M}_n}_{p}$ by:
$$\psi_k(a)=(W_1(a\otimes id_{\frac{n}{k}})W^*_1,W_2(a\otimes id_{\frac{n}{k}})W^*_2,...\,,W_p(a\otimes id_{\frac{n}{k}})W^*_p),$$ where

\hspace{1.0cm}$W_1=1\otimes id_{\frac{n}{k}},\;W_2=V^*\otimes
id_{\frac{n}{k}},\;...\,,\;W_p=(V^*)^{p-1}\otimes id_{\frac{n}{k}}$.

\noindent Then it is easy to check that $\psi_k$ intertwines
$\alpha_k$ and $\beta_n$, and $\psi_{k*}=F_k$. The natural extension
$\widetilde{\psi}_{k}$ intertwines $\hat{\alpha}_{k}$ and
$\hat{\beta}_{n}$, so $\widetilde{\psi}_{k*}=\phi_k$.

(3). $A_k=\underbrace{\textrm{M}_k\oplus...\oplus
\textrm{M}_k}_{p},\;B_n=\M_n$.

Since $F_k$ intertwines $\alpha_k$ and $\beta_n$, by calculation,
$F_k=\left(
       \begin{array}{cccc}
         \frac{n}{pk} &\frac{n}{pk}  & \cdots & \frac{n}{pk} \\
       \end{array}
     \right)
$. Since $\phi_k$ intertwines $\hat{\alpha}_{k*}$ and
$\hat{\beta}_{n*}$, we have $\phi_k=\left(
                               \begin{array}{c}
                                 l_1 \\
                                 l_1 \\
                                \vdots\\
                                 l_1 \\
                               \end{array}
                             \right)
$. Moreover, by the assumption of preserving the special elements,
one gets $l_1=\frac{n}{pk}$. Let $V$ be the unitary implementing the
automorphism on $\textrm{M}_n$. To define a homomorphism which
intertwines $\alpha_k$ and $\beta_n$, we need to find a unitary $W
$, such that Ad$(W^*VW)$ sends diag$(a_1\otimes
id_{\frac{n}{pk}},a_2\otimes id_{\frac{n}{pk}},...\,,a_p\otimes
id_{\frac{n}{pk}})$ to diag$(a_p\otimes id_{\frac{n}{pk}},a_1\otimes
id_{\frac{n}{pk}},...\,,a_{p-1}\otimes id_{\frac{n}{pk}})$ for all
$(a_1,...\,,a_p)$ in $\underbrace{\textrm{M}_k\oplus...\oplus
\textrm{M}_k}_{p}$. By Lemma IV.2 in \cite{HR}, this can be done.

(4). $A_k=\underbrace{\textrm{M}_k\oplus...\oplus
\textrm{M}_k}_{p},\;B_n=\underbrace{\textrm{M}_n\oplus...\oplus
\textrm{M}_n}_{p}$.

Since $F_k$ intertwines $\alpha_k$ and $\beta_n$, by calculation, we
have: $$F_k=\left(
                                                             \begin{array}{cccc}
                                                               l_{11} & l_{12} & \ldots & l_{1p} \\
                                                               l_{1p} & l_{11} & \ldots & l_{1p-1} \\
                                                               \cdots & \cdots &  \cdots& \cdots \\
                                                               l_{12} & l_{13} & \ldots & l_{11} \\
                                                             \end{array}
                                                           \right),$$
and $(l_{11}+l_{12}+...+l_{1p})k=n$. Similarly, we also have
$\phi_k=\frac{n}{k}$.

Define a homomorphism $\psi_k:
\underbrace{\textrm{M}_k\oplus...\oplus \textrm{M}_k}_{p}\rightarrow
\underbrace{\textrm{M}_n\oplus...\oplus \textrm{M}_n}_{p}$ by:
$$\psi_k(a_1,a_2,...\,,a_p)=(\psi^{1}_k(\cdot),\psi^{2}_k(\cdot),...\,,\psi^{p}_{k}(\cdot)),$$
where $(\cdot)$ is the abbreviation of $(a_1,a_2,...,a_p)\in
\underbrace{\textrm{M}_n\oplus...\oplus \textrm{M}_n}_{p}$, and

\hspace{1.0cm}$\psi^{1}_k(a_1,a_2,...\,,a_p)=$diag$(a_1\otimes
id_{l_{11}},a_2\otimes id_{l_{12}},...\,,a_p\otimes id_{l_{1p}})$,

\hspace{1.0cm}$\psi^{2}_k(a_1,a_2,...\,,a_p)=$diag$(a_2\otimes
id_{l_{11}},a_3\otimes id_{l_{12}},...\,,a_1\otimes id_{l_{1p}})$,

\hspace{1.0cm}$\psi^{3}_k(a_1,a_2,...\,,a_p)=$diag$(a_{3}\otimes
id_{l_{11}},a_4\otimes id_{l_{12}},...\,,a_{2}\otimes id_{l_{1p}})$,

\hspace{1.0cm}$,...\,,$

\hspace{1.0cm}$\psi^{p}_k(a_1,a_2,...\,,a_p)=$diag$(a_p\otimes
id_{l_{11}},a_1\otimes id_{l_{12}},...\,,a_{p-1}\otimes
id_{l_{1p}})$.

\noindent Then $\psi_k$ satisfies all the requirements.
\end{proof}
\begin{cor}\label{cor1} Let $(A_k,\alpha_k,Z_p)$ and $(B_n,\beta_n,Z_p)$ be two
finite dimensional \C- dynamical systems. Let $F_k$ be an ordered
group morphism from $(K_0(A_k),[1_{A_k}],\alpha_{k*})$ to
$(K_0(B_n),$ $[1_{B_n}],\beta_{n*})$. Let $\phi_k$ be an ordered
group morphism from
$(K_0(A_k\rtimes_{\alpha_k}Z_p),\hat{\alpha}_{k*})$ to
$(K_0(B_n\rtimes_{\beta_n}Z_p),\hat{\beta}_{n*})$, which preserves
the special elements, and the following diagram
$$
                          \begin{array}{ccc}
                            K_0(A_k) & \longrightarrow & K_0(A_k\rtimes_{\alpha_k}Z_p) \\
                            \Big\downarrow\!\;F_k &  & \Big\downarrow\!\;\phi_k \\
                            K_0(B_n) & \longrightarrow & K_0(B_n\rtimes_{\beta_n}Z_p) \\
                          \end{array}
$$ commutes. Then there exists a
homomorphism $\psi_k$ from $(A_k,\alpha_k,Z_p)$ to
$(B_n,\beta_n,Z_p)$, such that $\psi_{k*}=F_k$, and
$\widetilde{\psi}_{k*}=\phi_k$, where $\widetilde{\psi}_k$ is the
natural extension of $\psi_k$ to $A_k\rtimes_{\alpha_k}Z_p$.
\end{cor}
\begin{proof}
\end{proof}
\section{local uniqueness}\label{sect4}
In this section, we are going to establish the local uniqueness
theorem, namely, if two morphisms between the finite dimensional
\C-dynamical systems agree on the K-theoretic invariants, then they
are unitarily equivalent by an equivariant unitary, namely, a
unitary in the fixed point subalgebra of the codomain algebra.
\begin{thrm} Let $\phi_k$ and $\psi_k$ be two homomorphisms from the irreducible finite dimensional \C-dynamical system $(A_k, \alpha_k, Z_p)$ to $(B_n, \beta_n,
Z_p)$. Denote by $\widetilde{\phi}_k$ and $\widetilde{\psi}_k$ the
morphisms from $A_k\rtimes_{\alpha_k}Z_p$ to
$B_n\rtimes_{\beta_n}Z_p$ induced by $\phi_k$ and $\psi_k$,
respectively. If $\phi_{k*}=\psi_{k*}$ and
$\widetilde{\phi}_{k*}=\widetilde{\psi}_{k*}$, then there exists a
unitary $W$ in $B^{\beta_n}_n$, the fixed point subalgebra of $B_n$,
such that $\phi_k=AdW\circ\psi_k$.
\end{thrm}
\begin{proof} Again we are going to prove the theorem in four cases.

(1). $A_k=\textrm{M}_k, B_n=\textrm{M}_n$.

Let $U\in\textrm{M}_k$ and $V\in\textrm{M}_n$ be the two unitaries
which implement the action $\alpha_k$ and $\beta_n$, respectively.
Let $X$ and $Y$ be two unitaries in $B_n$ such that
$$\phi_k(a)=X(a\otimes id_{\frac{n}{k}})X^*, \psi_k(a)=Y(a\otimes
id_{\frac{n}{k}})Y^*, \forall a\in \textrm{M}_k.$$ Since each
$\phi_k$ and $\psi_k$ intertwines the actions $\alpha_k$ and
$\beta_n$, we have that
\begin{align*}&X(UaU^*\otimes id_{\frac{n}{k}})X^*=VX(a\otimes
id_{\frac{n}{k}})X^*V^*,\\
&Y(UaU^*\otimes id_{\frac{n}{k}})Y^*=VY(a\otimes
id_{\frac{n}{k}})Y^*V^*.
\end{align*}
Hence, \begin{align*}
&X^*V^*X(U\otimes id_{\frac{n}{k}})(a\otimes
id_{\frac{n}{k}})=(a\otimes id_{\frac{n}{k}})X^*V^*X(U\otimes
id_{\frac{n}{k}}),\\
&Y^*V^*Y(U\otimes id_{\frac{n}{k}})(a\otimes
id_{\frac{n}{k}})=(a\otimes id_{\frac{n}{k}})Y^*V^*Y(U\otimes
id_{\frac{n}{k}}),
\end{align*} Set $L=X^*V^*X(U\otimes
id_{\frac{n}{k}})$, $N=Y^*V^*Y(U\otimes id_{\frac{n}{k}})$, then $L$
and $N$ commute with $(a\otimes id_{\frac{n}{k}})$, for all
$a\in\M_k$, and $$L^{p}=X^*(V^*)^{p}X(U^p\otimes
id_{\frac{n}{k}})=I, N^p=I.$$ Note that $L,N$ commute with all
$(a\otimes id_{\frac{n}{k}})$, then $L,N$ belong to $I_k\otimes
\M_{\frac{n}{k}}$. Let $S$ and $R$ be two unitaries in $I_k\otimes
\M_{\frac{n}{k}}$ such that \begin{align*} &SLS^*=I_k\otimes
diag(\lambda_1,...,\lambda_{\frac{n}{k}}),\\
&RN\!R^*=I_k\otimes diag(\mu_1,...,\mu_{\frac{n}{k}}).
\end{align*} For $a_0+a_1U_{\rho}+...+a_{p-1}U_{\rho^{p-1}}\in
\M_k\rtimes_{\alpha_k}Z_p$, take $a_0=a_2=...=a_{p-1}=0$, and
$a_1=U^*$, by Lemma 2.1, we have that
\begin{align*}\tilde{\phi}_k(I,
e^{i\frac{2\pi}{p}}I,...,e^{i\frac{2(p-1)\pi}{p}}I)&=(X(U^*\otimes
id_{\frac{n}{k}})X^*V,e^{\frac{i2\pi}{p}}X(U^*\otimes
id_{\frac{n}{k}})X^*V,...)\\
&=(XL^*X^*,e^{\frac{i2\pi}{p}}XL^*X^*,...),\\
\tilde{\psi}_k(I,
e^{i\frac{2\pi}{p}}I,...,e^{i\frac{2(p-1)\pi}{p}}I)&=(Y(U^*\otimes
id_{\frac{n}{k}})Y^*V,e^{\frac{i2\pi}{p}}Y(U^*\otimes
id_{\frac{n}{k}})Y^*V,...)\\
&=(YN^*Y^*,e^{\frac{i2\pi}{p}}YN^*Y^*,...).
\end{align*} Since $\tilde{\phi}_{k*}=\tilde{\psi}_{k*}$, then there
exists a unitary $Z$ such that $XL^*X^*=ZYN^*Y^*Z^*$, hence,
$L=X^*ZYN^*Y^*Z^*X$, so
$\{\lambda_1,...,\lambda_{\frac{n}{k}}\}=\{\mu_1,...,\mu_{\frac{n}{k}}\}$.
Then there exists a unitary $\tilde{Z}\in
I_k\otimes\M_{\frac{n}{k}}$ such that $L=\tilde{Z}N\tilde{Z}^*$.
Hence, $$X^*V^*X(U\otimes
id_{\frac{n}{k}})=\tilde{Z}Y^*V^*Y(U\otimes
id_{\frac{n}{k}})\tilde{Z}^*,$$ which implies that
$VX\tilde{Z}Y^*=X\tilde{Z}Y^*V$. Therefore $X\tilde{Z}Y^*\in
B^{\beta_n}_{n}$, put $W=X\tilde{Z}Y^*$, then
$\phi_k=AdW\circ\psi_n$.

(2). $A_k=\M_k, B_n=\underbrace{\textrm{M}_n\oplus...\oplus
\textrm{M}_n}_p$.

Let $U$ be the unitary which implements the action $\alpha_k$,
namely, $\rho(a)=UaU^*$, for all $a\in\M_k$. Let $X_1,...,X_p$ be
the untaries in $\M_n$ such that $\phi_k(a)=(X_1a\otimes
id_{\frac{n}{k}}X^*_1, ..., X_pa\otimes id_{\frac{n}{k}}X^*_p )$,
for all $a\in \M_k$; let $Y_1,...,Y_p$ be the untaries in $\M_n$
such that $\psi_k(a)=(Y_1a\otimes id_{\frac{n}{k}}Y^*_1, ...,
Y_pa\otimes id_{\frac{n}{k}}Y^*_p )$, for all $a\in \M_k$. Since
each $\phi_k$ and $\psi_k$ intertwines $\alpha_k$ and $\beta_n$, we
obtain that:\begin{align*} (X_1(UaU^*\otimes
id_{\frac{n}{k}})X^*_1,\,& X_2(UaU^*\otimes
id_{\frac{n}{k}})X^*_2,...,X_p(UaU^*\otimes
id_{\frac{n}{k}})X^*_p)\\
&=(X_p(a\otimes id_{\frac{n}{k}})X^*_p,X_1(a\otimes
id_{\frac{n}{k}})X^*_1,...,X_{p-1}(a\otimes
id_{\frac{n}{k}})X^*_{p-1}).
\end{align*} Hence,
\begin{align*}
&X_1(UaU^*\otimes id_{\frac{n}{k}})X^*_1=X_p(a\otimes
id_{\frac{n}{k}})X^*_p,\\
&X_2(UaU^*\otimes id_{\frac{n}{k}})X^*_2=X_1(a\otimes
id_{\frac{n}{k}})X^*_1,\\
&,...,\\
&X_p(UaU^*\otimes id_{\frac{n}{k}})X^*_p=X_{p-1}(a\otimes
id_{\frac{n}{k}})X^*_{p-1}.
\end{align*}
This implies that
\begin{align*}
&X^*_pX_1(U\otimes id_{\frac{n}{k}})(a\otimes
id_{\frac{n}{k}})=(a\otimes
id_{\frac{n}{k}})X^*_pX_1(U\otimes id_{\frac{n}{k}}),\\
&X^*_1X_2(U\otimes id_{\frac{n}{k}})(a\otimes
id_{\frac{n}{k}})=(a\otimes
id_{\frac{n}{k}})X^*_1X_2(U\otimes id_{\frac{n}{k}}),\\
&,...,\\
&X^*_{p-1}X_p(U\otimes id_{\frac{n}{k}})(a\otimes
id_{\frac{n}{k}})=(a\otimes id_{\frac{n}{k}})X^*_{p-1}X_p(U\otimes
id_{\frac{n}{k}}).
\end{align*} Similarly, we also have:
\begin{align*}
&Y^*_pY_1(U\otimes id_{\frac{n}{k}})(a\otimes
id_{\frac{n}{k}})=(a\otimes
id_{\frac{n}{k}})Y^*_pY_1(U\otimes id_{\frac{n}{k}}),\\
&Y^*_1Y_2(U\otimes id_{\frac{n}{k}})(a\otimes
id_{\frac{n}{k}})=(a\otimes
id_{\frac{n}{k}})Y^*_1Y_2(U\otimes id_{\frac{n}{k}}),\\
&,...,\\
&Y^*_{p-1}Y_p(U\otimes id_{\frac{n}{k}})(a\otimes
id_{\frac{n}{k}})=(a\otimes id_{\frac{n}{k}})Y^*_{p-1}Y_p(U\otimes
id_{\frac{n}{k}}).
\end{align*} Our goal is to find a unitary $W=(W_1,...,W_p)\in B^{\beta_n}_n$,
such that $\phi_k=AdW\circ\psi_k$. Note that $W\in B^{\beta_n}_n$
means that $(W_1,W_2,...,W_p)=(W_p,W_1,...,W_{p-1})$, namely,
$W_1=W_2=...=W_p$.

Set \begin{align*} &L_1=X^*_pX_1(U\otimes id_{\frac{n}{k}}),
N_1=Y^*_pY_1(U\otimes id_{\frac{n}{k}}),\\
&,...,\\
&L_p=X^*_{p-1}X_p(U\otimes id_{\frac{n}{k}}),
N_p=Y^*_{p-1}Y_p(U\otimes id_{\frac{n}{k}}),
\end{align*} then by the calculation above, all of these $L_i,N_i$
commute with $a\otimes id_{\frac{n}{k}}$, $\forall a\in \M_k$. Then
$N_pL^*_p=Y^*_{p-1}Y_pX^*_pX_{p-1}$, which implies
$X_pY^*_p=X_{p-1}L_pN^*_pY^*_{p-1}$. Moreover, \begin{align*}
X_{p-2}L_{p-1}L_pN^*_pN^*_{p-1}Y^*_{p-2}&=X_{p-2}X^*_{p-2}X_{p-1}(U\otimes
id_{\frac{n}{k}})L_pN^*_p(U^*\otimes
id_{\frac{n}{k}})Y^*_{p-1}Y_{p-2}Y^*_{p-2}\\
&=X_{p-1}L_pN^*_pY^*_{p-1}.\end{align*} Similarly, we have
\begin{align*}
&X_{p-3}L_{p-2}L_{p-1}L_pN^*_pN^*_{p-1}N^*_{p-2}Y^*_{p-3}=X_{p-2}L_{p-1}L_pN^*_pN^*_{p-1}Y^*_{p-2},\\
&,...,\\
&X_1L_2...L_pN^*_p...N^*_2Y^*_1=X_2L_3...L_pN^*_p...N^*_3Y^*_2.
\end{align*} So all of these terms equal to
$X_{p-1}L_pN^*_pY^*_{p-1}=X_pY^*_p$.

Put
$$W=(X_1L_2...L_pN^*_p...N^*_2Y^*_1,X_2L_3...L_pN^*_p...N^*_3Y^*_2,...,X_pY^*_p),$$
then $W\in B^{\beta_n}_n$, and $\phi_k=AdW\circ\psi_n$, since for
each $i=1,...,p-1$, we have
$$X_{i}L_{i+1}...L_pN^*_p...N^*_{i+1}Y^*_{i}Y_{i}(a\otimes
id_{\frac{n}{k}})Y^*_{i}Y_{i}N_{i+1}...N_pL^*_p...L^*_{i+1}X^*_{i}=X_i(a\otimes
id_{\frac{n}{k}})X^*_i.$$

(3). $A_k=\underbrace{\textrm{M}_k\oplus...\oplus \textrm{M}_k}_{p},
B_n=\M_n$.

Let $V$ be the unitary which implements the action $\beta_n$,
namely, $\rho(a)=VaV^*$, for all $a\in\M_n$. Let $X,Y$ be unitaries
in $\M_n$ such that
\begin{align*}
&\phi(a_1,a_2,...,a_p)=Xdiag(a_1\otimes id_{\frac{n}{pk}},a_2\otimes
id_{\frac{n}{pk}},...,a_p\otimes id_{\frac{n}{pk}})X^*,\\
&\psi(a_1,a_2,...,a_p)=Ydiag(a_1\otimes id_{\frac{n}{pk}},a_2\otimes
id_{\frac{n}{pk}},...,a_p\otimes id_{\frac{n}{pk}})Y^*.
\end{align*} for all $(a_1,a_2,...,a_p)\in\underbrace{\textrm{M}_k\oplus...\oplus
\textrm{M}_k}_{p}$, This is the case since each $\phi_{*}$ and
$\psi_{*}$ intertwines the actions, and $\phi_{k*}=\psi_{k*}$.

Since each $\phi_k$ and $\psi_k$ intertwines the actions,  we obtain
that:
$$X\left(
     \begin{array}{cccc}
       a_p\otimes id_{\frac{n}{pk}} &  &  &  \\
        & a_1\otimes id_{\frac{n}{pk}} &  &  \\
        &  & \ddots &  \\
        &  &  & a_{p-1}\otimes id_{\frac{n}{pk}} \\
     \end{array}
   \right)
X^*$$ $$=VX\left(
     \begin{array}{cccc}
       a_1\otimes id_{\frac{n}{pk}} &  &  &  \\
        & a_2\otimes id_{\frac{n}{pk}} &  &  \\
        &  & \ddots &  \\
        &  &  & a_{p}\otimes id_{\frac{n}{pk}} \\
     \end{array}
   \right)X^*V^*.$$  and
$$Y\left(
     \begin{array}{cccc}
       a_p\otimes id_{\frac{n}{pk}} &  &  &  \\
        & a_1\otimes id_{\frac{n}{pk}} &  &  \\
        &  & \ddots &  \\
        &  &  & a_{p-1}\otimes id_{\frac{n}{pk}} \\
     \end{array}
   \right)
Y^*$$ $$=VY\left(
     \begin{array}{cccc}
       a_1\otimes id_{\frac{n}{pk}} &  &  &  \\
        & a_2\otimes id_{\frac{n}{pk}} &  &  \\
        &  & \ddots &  \\
        &  &  & a_{p}\otimes id_{\frac{n}{pk}} \\
     \end{array}
   \right)Y^*V^*.$$ Set $P=\left(
                             \begin{array}{cccc}
                                &  &  & I_{\frac{n}{k}} \\
                               I_{\frac{n}{k}} &  &  &  \\
                                & \ddots &  &  \\
                                &  & I_{\frac{n}{k}} &  \\
                             \end{array}
                           \right),
   $ then $$P\left(
     \begin{array}{cccc}
       a_p\otimes id_{\frac{n}{pk}} &  &  &  \\
        & a_1\otimes id_{\frac{n}{pk}} &  &  \\
        &  & \ddots &  \\
        &  &  & a_{p-1}\otimes id_{\frac{n}{pk}} \\
     \end{array}
   \right)P^*$$
\noindent $$=\left(
     \begin{array}{cccc}
       a_1\otimes id_{\frac{n}{pk}} &  &  &  \\
        & a_2\otimes id_{\frac{n}{pk}} &  &  \\
        &  & \ddots &  \\
        &  &  & a_{p}\otimes id_{\frac{n}{pk}} \\
     \end{array}
   \right). $$ So we have: $$XP^*P\left(
     \begin{array}{cccc}
       a_p\otimes id_{\frac{n}{pk}} &  &  &  \\
        & a_1\otimes id_{\frac{n}{pk}} &  &  \\
        &  & \ddots &  \\
        &  &  & a_{p-1}\otimes id_{\frac{n}{pk}} \\
     \end{array}
   \right)P^*PX^*$$ $$=VX\left(
     \begin{array}{cccc}
       a_1\otimes id_{\frac{n}{pk}} &  &  &  \\
        & a_2\otimes id_{\frac{n}{pk}} &  &  \\
        &  & \ddots &  \\
        &  &  & a_{p}\otimes id_{\frac{n}{pk}} \\
     \end{array}
   \right)X^*V^*,$$ namely, $$XP^*\left(
     \begin{array}{cccc}
       a_1\otimes id_{\frac{n}{pk}} &  &  &  \\
        & a_2\otimes id_{\frac{n}{pk}} &  &  \\
        &  & \ddots &  \\
        &  &  & a_{p}\otimes id_{\frac{n}{pk}} \\
     \end{array}
   \right)PX^*$$ $$=VX\left(
     \begin{array}{cccc}
       a_1\otimes id_{\frac{n}{pk}} &  &  &  \\
        & a_2\otimes id_{\frac{n}{pk}} &  &  \\
        &  & \ddots &  \\
        &  &  & a_{p}\otimes id_{\frac{n}{pk}} \\
     \end{array}
   \right)X^*V^*,$$ so $$X^*V^*XP^*\left(
     \begin{array}{cccc}
       a_1\otimes id_{\frac{n}{pk}} &  &  &  \\
        & a_2\otimes id_{\frac{n}{pk}} &  &  \\
        &  & \ddots &  \\
        &  &  & a_{p}\otimes id_{\frac{n}{pk}} \\
     \end{array}
   \right)$$ $$=\left(
     \begin{array}{cccc}
       a_1\otimes id_{\frac{n}{pk}} &  &  &  \\
        & a_2\otimes id_{\frac{n}{pk}} &  &  \\
        &  & \ddots &  \\
        &  &  & a_{p}\otimes id_{\frac{n}{pk}} \\
     \end{array}
   \right)X^*V^*XP^*,$$ similarly, $$Y^*V^*YP^*\left(
     \begin{array}{cccc}
       a_1\otimes id_{\frac{n}{pk}} &  &  &  \\
        & a_2\otimes id_{\frac{n}{pk}} &  &  \\
        &  & \ddots &  \\
        &  &  & a_{p}\otimes id_{\frac{n}{pk}} \\
     \end{array}
   \right)$$ $$=\left(
     \begin{array}{cccc}
       a_1\otimes id_{\frac{n}{pk}} &  &  &  \\
        & a_2\otimes id_{\frac{n}{pk}} &  &  \\
        &  & \ddots &  \\
        &  &  & a_{p}\otimes id_{\frac{n}{pk}} \\
     \end{array}
   \right)Y^*V^*YP^*.$$ Put $L=X^*V^*XP^*$, $N=Y^*V^*YP^*$, then
   $$L=diag(L_1,...,L_p), N=diag(N_1,...,N_p),$$ where each $L_i$ and $N_i$ belongs to $I_{k}\otimes \M_{\frac{n}{pk}}$, which means that each of them commutes
   with any matrix in $\M_{k}\otimes id_{\frac{n}{pk}}$. Hence, $$X^*V^*X=LP=\left(
                             \begin{array}{cccc}
                                &  &  & L_{1} \\
                               L_{2} &  &  &  \\
                                & \ddots &  &  \\
                                &  & L_{p} &  \\
                             \end{array}
                           \right),$$ and $$Y^*V^*Y=NP=\left(
                             \begin{array}{cccc}
                                &  &  & N_{1} \\
                               N_{2} &  &  &  \\
                                & \ddots &  &  \\
                                &  & N_{p} &  \\
                             \end{array}
                           \right).$$ Since $V^{p}=I$, we obtain $$\left(
                             \begin{array}{cccc}
                                &  &  & L_{1} \\
                               L_{2} &  &  &  \\
                                & \ddots &  &  \\
                                &  & L_{p} &  \\
                             \end{array}
                           \right)^p=I, \left(
                             \begin{array}{cccc}
                                &  &  & N_{1} \\
                               N_{2} &  &  &  \\
                                & \ddots &  &  \\
                                &  & N_{p} &  \\
                             \end{array}
                           \right)^p=I.$$ Form this we have that
                           \begin{align*}
                           &L_{1}L_{p}...L_{2}=I,\\
                           &L_{2}L_{1}...L_{3}=I,\\
                           &,...,\\
                           &L_{p}L_{p-1}...L_{1}=I,
                           \end{align*} similarly,
                           \begin{align*}
                           &N_{1}N_{p}...N_{2}=I,\\
                           &N_{2}N_{1}...N_{3}=I,\\
                           &,...,\\
                           &N_{p}N_{p-1}...N_{1}=I.
                           \end{align*} Set $$Z=\left(
                                         \begin{array}{ccccc}
                                           N_{1}L_{p}...L_{2} &  &  &  &  \\
                                            & N_{2}N_{1}L_{p}...L_{3} &  &  &  \\
                                            &  & \ddots &  &  \\
                                            &  &  & N_{p-1}...N_1L_{p} &  \\
                                            &  &  &  & I \\
                                         \end{array}
                                       \right),
                           $$ By the relations above, we
                           have that $$Z\left(
                             \begin{array}{cccc}
                                &  &  & L_{1} \\
                               L_{2} &  &  &  \\
                                & \ddots &  &  \\
                                &  & L_{p} &  \\
                             \end{array}
                           \right)Z^*=\left(
                             \begin{array}{cccc}
                                &  &  & N_{1} \\
                               N_{2} &  &  &  \\
                                & \ddots &  &  \\
                                &  & N_{p} &  \\
                             \end{array}
                           \right),$$ namely, $ZX^*V^*XZ^*=Y^*V^*Y$.
                           Put $W=XZ^*Y^*$, then $WV=VW$, so $W\in
                           B^{\beta_n}_{n}$, and $\phi_k=AdW\circ
                           \psi_n$.

(4). $A_k=\underbrace{\textrm{M}_k\oplus...\oplus \textrm{M}_k}_{p},
B_n=\underbrace{\textrm{M}_n\oplus...\oplus \textrm{M}_n}_{p}$.

Since $\phi_k$ intertwines the actions $\alpha_k$ and $\beta_n$, by
calculation, $$\phi_{k*}=\left(
                                                             \begin{array}{cccc}
                                                               l_{11} & l_{12} & \ldots & l_{1p} \\
                                                               l_{1p} & l_{11} & \ldots & l_{1p-1} \\
                                                               \cdots & \cdots &  \cdots& \cdots \\
                                                               l_{12} & l_{13} & \ldots & l_{11} \\
                                                             \end{array}
                                                           \right),$$ where $(l_{11}+l_{12}+...+l_{1p})k=n$, and $\phi_k$ is of the following form: $$
\phi_k(a_1,a_2,...,a_p)=(\phi_1(.),\phi_2(.),...,\phi_p(.)), \forall
(a_1,a_2,...,a_p)\in \underbrace{\textrm{M}_n\oplus...\oplus
\textrm{M}_n}_{p},$$ where $(.)$ is the abbreviation of
$(a_1,a_2,...,a_p)$, and
\begin{align*}
&\phi_1(a_1,a_2,...,a_p)=Xdiag(a_1\otimes id_{l_{11}},a_2\otimes
id_{l_{12}},...,a_p\otimes id_{l_{1p}})X^*,\\
&\phi_2(a_1,a_2,...,a_p)=Xdiag(a_2\otimes id_{l_{11}},a_3\otimes
id_{l_{12}},...,a_1\otimes id_{l_{1p}})X^*,\\
&,...,\\
&\phi_p(a_1,a_2,...,a_p)=Xdiag(a_p\otimes id_{l_{11}},a_1\otimes
id_{l_{12}},...,a_{p-1}\otimes id_{l_{1p}})X^*,
\end{align*} here $X$ is a unitary in $\M_n$. Since
$\phi_{k*}=\psi_{k*}$, similarly, $$
\psi_k(a_1,a_2,...,a_p)=(\psi_1(.),\psi_2(.),...,\psi_p(.)), \forall
(a_1,a_2,...,a_p)\in \underbrace{\textrm{M}_n\oplus...\oplus
\textrm{M}_n}_{p},$$ where $(.)$ is the abbreviation of
$(a_1,a_2,...,a_p)$, and
\begin{align*}
&\psi_1(a_1,a_2,...,a_p)=Ydiag(a_1\otimes id_{l_{11}},a_2\otimes
id_{l_{12}},...,a_p\otimes id_{l_{1p}})Y^*,\\
&\psi_2(a_1,a_2,...,a_p)=Ydiag(a_2\otimes id_{l_{11}},a_3\otimes
id_{l_{12}},...,a_1\otimes id_{l_{1p}})Y^*,\\
&,...,\\
&\psi_p(a_1,a_2,...,a_p)=Ydiag(a_p\otimes id_{l_{11}},a_1\otimes
id_{l_{12}},...,a_{p-1}\otimes id_{l_{1p}})Y^*,
\end{align*} here $Y$ is a unitary in $\M_n$.

Put $W=(XY^*,...,XY^*)\in B^{\beta_n}_n$, then it is clear that
$\phi_k=AdW\circ\psi_k.$
\end{proof}
\begin{cor}\label{cor2} Let $\phi_k$ and $\psi_k$ be two homomorphisms from the finite dimensional \C-dynamical system $(A_k, \alpha_k, Z_p)$ to $(B_n, \beta_n,
Z_p)$. Denote by $\widetilde{\phi}_k$ and $\widetilde{\psi}_k$ the
morphisms from $A_k\rtimes_{\alpha_k}Z_p$ to
$B_n\rtimes_{\beta_n}Z_p$ induced by $\phi_k$ and $\psi_k$,
respectively. If $\phi_{k*}=\psi_{k*}$ and
$\widetilde{\phi}_{k*}=\widetilde{\psi}_{k*}$, then there exists a
unitary $W$ in $B^{\beta_n}_n$, the fixed point subalgebra of $B_n$,
such that $\phi_k=AdW\circ\psi_k$.
\end{cor}
\section{classification}\label{sect5}
In this section, we prove the classification \textbf{Theorem}
\ref{main} by the equivariant Elliott's intertwining argument.
\begin{proof} First of all, by standard argument, the K-theoretic invariants of the AF \C-dynamical systems can be lifted to finite stages, namely, by passing
to subsequences and changing notation, we could obtain the following
intertwinings: $$ \xymatrix{
(\textrm{K}_0(A_1),\alpha_{1_*}) \ar[r] \ar[d] & (\textrm{K}_0(A_2),\alpha_{2_*}) \ar[r] \ar[d] & \cdots \ar[r] & (\textrm{K}_0(A), \alpha_{*}) \ar@{<=>}[d]\\
(\textrm{K}_0(B_1),\beta_{1_*}) \ar[r] \ar[ur] &
(\textrm{K}_0(B_2),\beta_{2_*}) \ar[r] \ar[ur] & \cdots \ar[r] &
(\textrm{K}_0(B),\beta_{*}) }
$$ and $$
\xymatrix{
(\textrm{K}_0(A_1\rtimes_{\alpha_1}Z_p),\hat{\alpha}_{1_*}) \ar[r] \ar[d] & (\textrm{K}_0(A_2\rtimes_{\alpha_2}Z_p),\hat{\alpha}_{2_*}) \ar[r] \ar[d] & \cdots \ar[r] & (\textrm{K}_0(A\rtimes_{\alpha}Z_p), \hat{\alpha}_{*}) \ar@{<=>}[d]\\
(\textrm{K}_0(B_1\rtimes_{\beta_1}Z_p),\hat{\beta}_{1_*}) \ar[r]
\ar[ur] & (\textrm{K}_0(B_2\rtimes_{\beta_2}Z_p),\hat{\beta}_{2_*})
\ar[r] \ar[ur] & \cdots \ar[r] &
(\textrm{K}_0(B\rtimes_{\beta}Z_p),\hat{\beta}_{*}). }
$$
Second, we would like to make the two intertwinings above to be
compatible. Note that we have that: $$
  \begin{array}{ccccccccc}
    \textrm{K}_0(A_1)&\hskip-1em \rightarrow &\hskip-1em \textrm{K}_0(A) &\hskip-1em \cong &\hskip-1em \textrm{K}_0(B) &\hskip-1em \leftarrow &\hskip-1em \cdots & \hskip-1em \leftarrow & \hskip-1em\textrm{K}_0(B_1) \\
    \downarrow &  &\hskip-1em \downarrow &  & \hskip-1em\downarrow &  &  &  & \hskip-1em\downarrow \\
    \textrm{K}_0(A_1\rtimes_{\alpha_1}Z_p) & \hskip-1em\rightarrow &\hskip-1em \textrm{K}_0(A\rtimes_{\alpha}Z_p) & \hskip-1em\cong &\hskip-1em \textrm{K}_0(B\rtimes_{\beta}Z_p) & \hskip-1em\leftarrow & \hskip-1em\cdots & \hskip-1em\leftarrow & \hskip-1em\textrm{K}_0(B_1\rtimes_{\beta}Z_p) \\
  \end{array},
$$ Hence, there exists $n$, such that the following diagram $$\xymatrix{
\textrm{K}_0(A_1) \ar[r] \ar[d] & \textrm{K}_0(B_n) \ar[d]\\
\textrm{K}_0(A_1\rtimes_{\alpha_1}Z_p \ar [r]) &
\textrm{K}_0(B_n\rtimes_{\beta_n}Z_p)
 }$$ commutes. After reindexing, the two intertwinings above could
 satisfy the following commutative diagrams:$$\xymatrix{\textrm{K}_0(A_n) \ar[r] \ar[d] & \textrm{K}_0(B_n) \ar[d]\\
\textrm{K}_0(A_n\rtimes_{\alpha_n}Z_p \ar [r]) &
\textrm{K}_0(B_n\rtimes_{\beta_n}Z_p),}$$ and $$\xymatrix{\textrm{K}_0(B_n) \ar[r] \ar[d] & \textrm{K}_0(A_{n+1}) \ar[d]\\
\textrm{K}_0(B_n\rtimes_{\beta_n}Z_p \ar [r]) &
\textrm{K}_0(A_{n+1}\rtimes_{\alpha_{n+1}}Z_p).}$$ Also, these
intertwinings can preserve the special elements and the units.

Now, we can apply the local existence and the local uniqueness
results on finite stages. By Corollary \ref{cor1}, we can lift each
morphism of the invariant to a morphism between the dynamical
systems. By Corollary \ref{cor2}, we can correct each morphism by an
inner morphism commuting with the actions, so we obtain an
intertwining of the dynamical systems: $$ \xymatrix{
(A_1,\alpha_1) \ar[r] \ar[d] & (A_2,\alpha_2) \ar[r] \ar[d] & \cdots \ar[r] & (A, \alpha) \ar@{-->} [d]^{\psi} \\
(B_1,\beta_1) \ar[r] \ar[ur] & (B_2,\beta_2) \ar[r] \ar[ur] & \cdots
\ar[r] & (B,\beta). }$$ Hence, $(A,\alpha)$ and $(B,\beta)$ are
isomorphic by an isomorphism $\psi$ which induces $F$ and $\phi$.
\end{proof}

An immediate corollary of this theorem is the classification of
inductive limit actions of $Z_p$ on AF algebras.

\begin{cor} Let $A$ be an AF algebra, and $\alpha, \beta$ be two
inductive limit group actions of $Z_p$ on $A$, then $\alpha$ and
$\beta$ are conjugate each other if and only if \begin{align*}
&(K_0(A), [1_A], \alpha_*)\cong(K_0(A), [1_A], \beta_*)\\
&(K_0(A\rtimes_{\alpha}Z_p), \zeta_{\alpha},
\hat{\alpha}_{*})\cong(K_0(A\rtimes_{\beta}Z_p), \zeta_{\beta},
\hat{\beta}_{*})
\end{align*} and these two isomorphisms are compatible with natural embedding of the K-theory of the algebra to the K-theory of the cross product.
\end{cor}

\proof[Acknowledgements] This paper was initiated while the first
author was a postdoctoral fellow at the Fields Institute; and was
completed after he moved to the Research Center for Operator
Algebras, East China Normal University (ECNU). He is very grateful
to Professor George A. Elliott for his support and to the Fields
Institute for its hospitality; he also appreciates Professor Huaxin
Lin and ECNU's support. The second author is supported by a NSFC grant in China.


\begin{thebibliography}{}
\bibitem{Bla} B. Blackadar, \emph{K-theory for operator algebras.}
Mathematical Sciences Research Institute Publications, 5,
Springer-Verlag, New York, 1986.



\bibitem{BEEK} O. Bratteli, G. A. Elliott, D. E. Evans and A. Kishimoto, \emph{on the classification of inductive limits of inner actions of a compact group.} Current topics in opertor algebras, Nara,
1990, 13--24, World Scientic Publishing, River Edge, New Jersey,
1991.

\bibitem {Ell} G. A. Elliott, \emph{On the classification of induvtive limits of sequences of semisimple finite-dimensional algebras.} J. Algebra 38 (1976), 29--44.

\bibitem {ES} G. A. Elliott and H. Su, \emph{\K-theoretic classification for inductive limit $Z_2$ actions on AF algebras.} Canad. J. Math. 48 (1996), 946--958.




\bibitem {HR} D. E. Handelman and W. Rossmann, \emph{Actions of compact groups on AF \Ca s.}
 Illinois J. Math. 29 (1985), 51--95.

\bibitem {HWZ} I. Hirshberg, W. Winter, and J. Zacharias, \emph{Rokhlin dimension and
C*-dynamics.} preprint arXiv: 1209.1618v1.

\bibitem {izumi1} M. Izumi, \emph{Finite group actions on C*-algebras with the Rokhlin property I.} Duke Math. 1222 (2004), no. 2, 233--280.

\bibitem {izumi2} M. Izumi, \emph{Finite group actions on C*-algebras with the Rokhlin property II.} Adv. Math. 184 (2004), no. 1, 119--160.

\bibitem {K} A. Kishimoto, \emph{Actions of finite groups on AF \Ca .} Internat. J. Math. 1 (1990), 267--292.

\bibitem {Phi} N. C. Phillips, \emph{Finite cyclic group actions with the tracial Rokhlin
property.} Trans. Amer. Math. Soc., to appear (arXiv: math.
OA/0609785).

\bibitem {Phi1} N. C. Phillips, \emph{The tracial Rokhlin property is
generic.} preprint arXiv: 1209.3859v1.

\bibitem{Su} H. Su, \emph{\K-theoretic classification for certain indutive limit $Z_2$ actions on real rank zero \Ca s.} Trans. Amer. Math. Soc. 348 (1996), no. 10, 4199--4230.








\end{thebibliography}
\end{document}